\documentclass[12pt]{elsarticle}
\usepackage{mathrsfs,mathtools}
\usepackage[left=1 in,top=1 in,right=1 in,bottom=1 in]{geometry}

\usepackage{hyperref}

\usepackage{url}
\usepackage{amssymb}
\usepackage{amsmath}
\usepackage{amsthm}
\usepackage{enumerate}
\newtheorem{theorem}{Theorem}[section]

\newtheorem{corollary}[theorem]{Corollary}
\newtheorem{lemma}[theorem]{Lemma}
\newtheorem{proposition}[theorem]{Proposition}

\newtheorem{definition}[theorem]{Definition}

\def\R{{\mathbb R}}
\def\N{{\mathbb N}}
\def\Z{{\mathbb Z}}
\def\d{{\rm d}}

\def\:{\colon}
\def\<{{\langle}}
\def\>{{\rangle}}

\def\A{{\mathscr A}}
\def\AA{{\mathbb A}}
\def\B{{\mathscr B}}
\def\BB{{\mathbb B}}

\numberwithin{equation}{section}

\begin{document}

\begin{frontmatter}

\author[LH]{Luan T. Hoang}
\ead{luan.hoang@ttu.edu}
\address[LH]{Department of Mathematics and Statistics,
Texas Tech University\\
Box 41042,
Lubbock, TX 79409-1042, USA.}
\author[Reno]{Eric J. Olson}
\ead{ejolson@unr.edu}
\address[Reno]{Department of
  Mathematics/084, University of Nevada\\ Reno, NV 89557, USA.}
\author[warwick]{James C. Robinson}
\ead{J.C.Robinson@warwick.ac.uk}
\address[warwick]{Mathematics Institute, Zeeman Building,\\
	University of Warwick, Coventry CV4 7AL, UK.}

\date{\today}

\title{Continuity of pullback and uniform attractors}

\begin{abstract}
We study the continuity of pullback and uniform attractors for
non-autonomous dynamical systems with respect to perturbations
of a parameter.
Consider a family of dynamical systems parameterized
by $\lambda\in\Lambda$, where $\Lambda$ is a complete metric space, such that for
each $\lambda\in\Lambda$ there exists a
unique pullback attractor $\A_\lambda(t)$.
Using the theory of Baire category we show
under natural conditions
that there exists a
residual set $\Lambda_*\subseteq\Lambda$ such that
for every $t\in\R$
the function $\lambda\mapsto\A_\lambda(t)$ is continuous
at each $\lambda\in\Lambda_*$
with respect to the Hausdorff metric.
Similarly, given a family of uniform attractors $\AA_\lambda$,
there is a residual set at which
the map $\lambda\mapsto\AA_\lambda$ is continuous.
We also introduce notions of equi-attraction suitable for
pullback and uniform attractors and then
show when $\Lambda$ is compact that
the continuity of pullback attractors and
uniform attractors with respect to $\lambda$ is equivalent to
pullback equi-attraction and, respectively,
uniform equi-attraction.
These abstract results are then illustrated in the context of the
Lorenz equations and the
two-dimensional Navier--Stokes equations.
\end{abstract}


\end{frontmatter}

\section{Introduction}

The theory of attractors plays an important role in understanding the
long time behavior of dynamical systems, see  Babin and Vishik \cite{BV},
Billotti and LaSalle \cite{BLS},  Chueshov \cite{Chu}, Hale \cite{Hale},
Ladyzhenskaya \cite{Ladyz}, Robinson \cite{JCR} and Temam \cite{Temam}.
For the autonomous theory, we consider a family of dissipative dynamical
systems parameterized by $\Lambda$ such that for each $\lambda\in\Lambda$
the corresponding dynamical system possesses a unique compact global
attractor $\A_\lambda\subseteq Y$, where $Y$ is a complete metric space
with metric $d_Y$.
Under very mild assumptions (see for example~\cite{HLR} and
the references therein)
the map $\lambda\mapsto \A_\lambda$ is known to be upper semicontinuous.
This means that
$$
\rho_Y(\A_\lambda,\A_{\lambda_0}) \to 0
\qquad\hbox{as}\qquad
\lambda\to\lambda_0
$$
where $\rho_Y(A,C)$ denotes the Hausdorff
semi-distance
\begin{equation}\label{defsemidist}
\rho_Y(A,C)=\adjustlimits\sup_{a\in A}\inf_{c\in C}d_Y(a,c).
\end{equation}
However, lower
semicontinuity
$$
\rho_Y(\A_{\lambda_0},\A_\lambda) \to 0
\qquad\hbox{as}\qquad
\lambda\to\lambda_0,
$$
and hence full continuity with respect to the Hausdorff metric,
is much harder to prove.

For autonomous systems, general results on lower semicontinuity require strict conditions on
the structure of the unperturbed global attractor, which are rarely satisfied
for complicated systems (see Hale and Raugel
\cite{HR} and Stuart and Humphries \cite{SH}). 
However,
Babin and Pilyugin \cite{BPY} and Luan et al.\  \cite{HOR1} showed, using the theory of Baire category, that
continuity holds for $\lambda_0$ in a residual
set $\Lambda_*\subseteq\Lambda$ under natural conditions
when $\Lambda$ is a complete metric space.
We recall this result for autonomous systems
as Theorem~\ref{auto} below.

Let $\Lambda$ and $X$ be complete metric spaces.  We will suppose that
$S_\lambda(\cdot)$
is a parameterized family of semigroups on $X$
for $\lambda\in \Lambda$ that satisfies the following properties:
\begin{itemize}
\item[(G1)] $S_\lambda(\cdot)$ has a global attractor $\A_\lambda$ for every $\lambda\in\Lambda$;
\item[(G2)] there is a bounded subset $D$ of $X$ such that $\A_\lambda\subseteq D$ for every $\lambda\in\Lambda$; and
\item[(G3)] for $t>0$, $S_\lambda(t)x$ is continuous in $\lambda$, uniformly for $x$ in bounded subsets of $X$.
\end{itemize}
Note that condition (G2) can be strengthened and (G3) weakened by
replacing {\it bounded\/} by {\it compact}.
These modified conditions will be referred to as
conditions (G2$'$) and (G3$'$).

\begin{theorem}
\label{auto}
  Under assumptions {\rm(G1--G3)} above---or under
  the assumptions {\rm(G1)}, {\rm(G2$'$)}
  and {\rm(G3$'$)}---$\A_\lambda$ is continuous in $\lambda$ at all
  $\lambda_0$ in a residual subset of $\Lambda$.  In particular the set
  of continuity points of $\A_\lambda$ is dense in $\Lambda$.

\end{theorem}

The proof developed
in \cite{HOR1}
of the above theorem,
which appears there as Theorem 5.1,
is more direct than
previous proofs (e.g.\ in \cite{BPY}) and can be modified
to establish analogous results for
the pullback attractors and uniform attractors of non-autonomous systems.
This is the main purpose of the present paper.
After briefly introducing some definitions and notations concerning
attractors and Baire category theory in Section~\ref{prelimsec},
in Section~\ref{pullbacksec} we prove
Theorem~\ref{thm23}, our main result concerning pullback
attractors.
Section~\ref{uniformsec} then contains
Theorem~\ref{thm31}, which provides similar results for uniform attractors.
In addition, we investigate the continuity of pullback and uniform
attractors on the entire parameter space $\Lambda$.
It was proved by Li and Kloeden
\cite{KL} (see also \cite{HOR1}) that when $\Lambda$
is compact, the continuity of the global attractors on $\Lambda$ is
equivalent to equi-attraction of the semigroups.
In Section~\ref{equi}, we extend this result and the notion of
equi-attraction to non-autonomous and uniform attractors.
Theorem~\ref{equiback} shows for pullback
attractors that continuity is equivalent to pullback equi-attraction,
while Theorem~\ref{equiuni} shows for uniform attractors that
continuity is equivalent to uniform equi-attraction.

We note that the continuity of pullback attractors is investigated
by Carvalho et al.\ \cite{CLR}, who extend the autonomous results to
non-autonomous systems, under strong conditions on the structure of the
pullback attractors.
Similarly, the notion of equi-attraction defined in Section~\ref{equi} is
a difficult property to discern for any concrete family of dynamical system.
In contrast, the continuity results from Sections~\ref{pullbacksec}
and~\ref{uniformsec} only require standard conditions
that are met in many applications.
We demonstrate this in Section~\ref{apps} with the Lorenz system of ODEs
and the two-dimensional Navier--Stokes equations.

\section{Preliminaries}\label{prelimsec}

We begin by setting our notation and recalling the definition of the Hausdorff metric.
Given a metric space $(Y,d_Y)$, denote by $B_Y(y,r)$ the ball of radius
$r$ centered at $y$,
$$
B_Y(y,r)=\{y\in Y:\ d_Y(x,y)<r\}.
$$
Write
$\Delta_Y$ for the symmetric Hausdorff distance
\begin{equation}
\Delta_Y(A,C)=\max(\rho_Y(A,C),\rho_Y(C,A))
\end{equation}
where $\rho_Y$ is the semi-distance between two
subsets $A$ and $C$ of $Y$ defined in (\ref{defsemidist}).
Denote by ${\it CB}(Y)$ the collection of all non-empty closed,
bounded subsets of a metric space $Y$, which is
itself a metric space with metric given by
the symmetric Hausdorff distance $\Delta_Y$.

In the same way that a continuous semigroup may be used to describe
an autonomous dynamical system, the concept of a non-autonomous
process may be used to describe a non-autonomous dynamical system.

\begin{definition}\label{processdef}
Let $(X,d_X)$ be a complete metric space. A \emph{process} $S(\cdot,\cdot)$ on $X$ is a two-parameter family of maps $S(t,s)\colon X\mapsto X$, $s\in\R$, $t\ge s$, such that
\begin{enumerate}
\item[{\rm(P1)}] $S(t,t)={\rm id}$;
\item[{\rm(P2)}] $S(t,\tau)S(\tau,s)=S(t,s)$ for all $t\ge\tau\ge s$; and
\item[{\rm(P3)}] $S(t,s)x$ is continuous in $x$, $t$, and $s$.
\end{enumerate}
\end{definition}

Given a non-autonomous process, there are two common ways to
characterize its asymptotic behavior:
roughly speaking,
the limit of $S(t,s)$ for a fixed $t$ as $s\to-\infty$ leads to
the definition of the pullback attractor, while the limit of
$S(t+s,s)$ as $t\to \infty$ leads to the uniform attractor
(given sufficient uniformity in $s$).
While both methods give rise to the same object for
autonomous dynamics, they may be different in the non-autonomous
case.

We begin with a formal definition of the pullback attractor, obtained by taking the limit as $s\to-\infty$.

\begin{definition}
A family of compact sets
$\A(\cdot)=\{\A(t)\colon t\in\R\}$ in $X$ is the
\emph{pullback attractor} for the process $S(\cdot,\cdot)$ if
\begin{enumerate}
\item[{\rm(A1)}] $\A(\cdot)$ is invariant: $S(t,s)\A(s)=\A(t)$ for all $t\ge s$;
\item[{\rm(A2)}] $\A(\cdot)$ is pullback attracting: for any bounded set $B$ in $X$ and $t\in\R$
$$
\rho_X\big(S(t,s)B,\A(t)\big)\to0\qquad\mbox{as}\qquad s\to-\infty;
$$
and
\item[{\rm(A3)}] $\A(\cdot)$ is minimal, in the sense that if $C(\cdot)$ is any other family of compact sets that satisfies (A1) and (A2) then $\A(t)\subseteq C(t)$ for all $t\in\R$.
\end{enumerate}
\end{definition}

The uniform attractor is obtained by taking the limit as $t\to\infty$.

\begin{definition}
A set $\mathbb A\subseteq X$ is the \emph{uniform attractor} if it is the minimal compact set such that
\begin{equation}\label{unicv}
\adjustlimits\lim_{t\to\infty}
	\sup_{s\in\R}\rho_X\big(S(t+s,s)B,{\mathbb A}\big)=0
\end{equation}
for any bounded $B\subseteq X$.
\end{definition}

We finish this section by stating
a few basic facts from the theory of Baire category
including an abstract residual continuity result.
Recall that a set is \emph{nowhere dense} if its closure
contains no non-empty open sets, and a set is \emph{residual} if its complement is the countable union
of  nowhere dense  sets.
It is a well-known fact that any residual subset of a
complete metric space is dense.

The following result,  an abstract version of Theorem 7.3 in Oxtoby \cite{Oxtoby}, was proved as Theorem 5.1 in \cite{HOR1}, and forms a key part of our proofs.

  \begin{theorem}\label{BCT}
Let $f_n\colon\Lambda\mapsto Y$ be a continuous map
  for each $n\in\N$, where $\Lambda$ is a complete metric space and $Y$ is
  any metric space. If $f$ is the pointwise limit of $f_n$, that is, if
$$
	f(\lambda)=\lim_{n\to\infty}f_n(\lambda)
\qquad\hbox{for each}\qquad\lambda\in\Lambda
$$
(implicit in this is the requirement that the limit exists) then the points of continuity of $f$ form a
  residual subset of $\Lambda$.

\end{theorem}

\section{Residual continuity of pullback attractors}\label{pullbacksec}

In this section we consider the continuity of pullback attractors.
Let $\Lambda$ be a complete metric space and $S_\lambda(\cdot,\cdot)$ a parameterized family of processes on $X$ with $\lambda\in \Lambda$.
Suppose that
\begin{itemize}
\item[(L1)] $S_\lambda(\cdot,\cdot)$ has a pullback attractor $\A_\lambda(\cdot)$ for every $\lambda\in\Lambda$;
\item[(L2)] there is a bounded subset $D$ of $X$ such that $\A_\lambda(t)\subseteq D$ for every $\lambda\in\Lambda$ and every $t\in\R$; and
\item[(L3)] for every $s\in\R$ and $t\ge s$, $S_\lambda(t,s)x$ is continuous in $\lambda$, uniformly for $x$ in bounded subsets of $X$.
\end{itemize}
We denote by (L2$'$) and (L3$'$) the assumptions (L2) and (L3), respectively, with {\it bounded}  replaced by {\it compact}.

The following result is proved for the autonomous case as Lemma 3.1 in \cite{HOR1}; we omit the proof for the non-autonomous case, which is identical.

\begin{lemma}\label{Dncts}
Assume either {\rm(L2)} and {\rm (L3)}, or {\rm(L2$'$)} and {\rm
(L3$'$)}.
Then for any $s\in\R$ and $t\ge s$, the map $\lambda\mapsto
\overline{S_\lambda(t,s)D}$ is continuous from $\Lambda$ into ${\it
CB}(X)$.
\end{lemma}

We also need the following related continuity result for
$S_\lambda(t,s)B$. Note that the result only treats sets $B\in {\it
CB}(K)$ for some compact $K$, which is crucial to the proof.

\begin{lemma}\label{jointcts}
Assume that
{\rm(L3$'$)} holds, and let $K$ be any compact subset of
$X$. Then for any $t\ge s$, the mapping $(\lambda,B)\mapsto S_\lambda(t,s)B$
is (jointly) continuous in $(\lambda,B)\in \Lambda\times  {\it CB}(K)$.
\end{lemma}

\begin{proof}
Since every $B\in {\it CB}(K)$ is compact and
$S_\lambda(t,s)x$ is continuous in $x$, it follows that the
image $S_\lambda(t,s)B$ is compact too.
Now suppose that $s\in\R$, $t\ge s$,
$\lambda_0\in\Lambda$, $B_0\in {\it CB}(K)$ and $\epsilon>0$.
Condition~(L3$'$) ensures that there exists a $\delta_1\in (0,1)$ such that
$$
d_\Lambda(\lambda_0,\lambda)<\delta_1
\quad\hbox{implies that}\quad
d_X\big(S_\lambda(t,s)x,S_{\lambda_0}(t,s)x\big)<\epsilon/2
\quad\mbox{for every}\quad
x\in K.
$$
Since $K$ is compact, the map $x\mapsto S_{\lambda_0}(t,s)x$ is
uniformly continuous on $K$;
in particular, there is a $\delta_2\in(0,1)$ such that
$$
d_X(x,y)<\delta_2\quad\hbox{with}\quad x,y\in K
\quad\hbox{implies that}\quad
d_X\big(S_{\lambda_0}(t,s)x,S_{\lambda_0}(t,s)y\big)<\epsilon/2.
$$
Set $\delta=\min\{\delta_1,\delta_2\}$.

Take $\lambda\in \Lambda$ with $d_\Lambda(\lambda,\lambda_0)<\delta$ and  $B\in {\it CB}(K)$ with $\Delta_X(B,B_0)<\delta$.
For any $b\in B$ there is a $b_0\in B_0$ such that $d_X(b,b_0)<\delta$.
Therefore
\begin{align}
d_X\big(S_\lambda(t,s)b,S_{\lambda_0}(t,s)b_0\big)
&\le d_X\big(S_\lambda(t,s)b,S_{\lambda_0}(t,s)b\big)\nonumber\\
&\qquad\qquad+ d_X\big(S_{\lambda_0}(t,s)b,S_{\lambda_0}(t,s)b_0\big)
\nonumber\\
&<\epsilon/2+\epsilon/2=\epsilon.\nonumber
\end{align}
Hence
\begin{equation}\label{r1}
\rho_X\big(S_\lambda(t,s)B,S_{\lambda_0}(t,s)B_0\big)\le \epsilon.
\end{equation}
The hypothesis $\Delta_X(B,B_0)<\delta$ also implies that
for any $b_0\in B_0$, there is $b\in B$ such that
$d_X(b,b_0)<\delta$.  Consequently
\begin{equation}\label{r2}
\rho_X\big(S_\lambda(t,s)B_0,S_{\lambda_0}(t,s)B\big)\le \epsilon.
\end{equation}
Combining \eqref{r1} and \eqref{r2} now yields
$\Delta_X\big(S_\lambda(t,s)B,S_{\lambda_0}(t,s)B_0\big)\le \epsilon$, which
proves the joint continuity as claimed.
  \end{proof}

As (L3) is a stronger hypothesis than (L3$'$)
we note that Lemma~\ref{jointcts} also holds under~(L3).
We now use Lemma~\ref{jointcts} to prove
the residual continuity
of pullback attractors.

\begin{theorem}\label{thm23}
Let $S_\lambda(\cdot,\cdot)$ be a family of processes
on $(X,d)$ each satisfying {\rm(P1--P3)}
and suppose that {\rm(L1)} holds along with either
\begin{itemize}
\item[\rm (i)] {\rm(L2$'$)} and {\rm(L3$'$)}, or
\item[\rm (ii)] {\rm(L2)}, {\rm(L3)}, and
for any  $\lambda_0\in\Lambda$ and  $t\in\R$,
there exists  $\delta>0$ such that
\begin{equation}\label{E}
\overline{\bigcup_{B_\Lambda (\lambda_0,\delta) }  \A_\lambda(t) }
\qquad\mbox{is compact}.
\end{equation}
\end{itemize}
Then, there exists a residual set $\Lambda_*$ in $\Lambda$ such that
for every $t\in \R$ the function $\lambda \mapsto \A_\lambda(t)$ is
continuous at each $\lambda\in \Lambda_*$.
\end{theorem}

\begin{proof}
From Lemma~\ref{Dncts} it follows that for each $n\in\Z$ and $s<n$
the function $\lambda\mapsto\overline{S_\lambda(n,s)D}$ is continuous.
Moreover, since by either (L2) or (L2$'$) we have $D\supseteq\A_\lambda(s)$,
then from the invariance
of the attractor (A1) it follows that
\begin{equation}\label{SA0}
\overline{S_\lambda(n,s)D}\supseteq S_\lambda(n,s)D\supseteq
S_\lambda(n,s)\A_\lambda(s)=\A_\lambda(n)
\qquad\mbox{for every}\qquad s\le n.
\end{equation}
Therefore,
the pullback attraction property (A2) yields
 \begin{equation}\label{Alan}
  \A_\lambda(n)=
	\lim_{s\to-\infty}\overline{S_\lambda(n,s)D},
  \end{equation}
where the convergence is with respect to the Hausdorff metric.
It follows from Theorem~\ref{BCT} that there is a residual set $\Lambda_n$
of $\Lambda$ at which the map $\lambda\in\Lambda \mapsto\A_\lambda(n)$
is continuous.
Since the countable intersection of residual sets is still residual,
then $\Lambda_* = \bigcap_{n\in\Z}\Lambda_n$ is a residual set
at which $\lambda\mapsto \A_\lambda(n)$
is continuous for every $n\in\Z$.

We now use the invariance of $\A_\lambda(\cdot)$ to obtain continuity
for every $t\in\R$.
For $t\notin\Z$ there is $n\in\Z$ such that $t\in(n,n+1)$.
Moreover,
\begin{equation}\label{inter}
	\A_\lambda(t)=S_\lambda(t,n)\A_\lambda(n).
\end{equation}
In case (i) set $K=D$; in case (ii) define
$$
	K=\overline{\bigcup_{B_\Lambda (\lambda_0,\delta)} \A_\lambda(t)}
$$
where $\delta>0$ has been chosen by (\ref{E})
such that $K$ is compact.
Since $\A_\lambda(n)$ is continuous at $\lambda\in \Lambda_*$
and $\A_\lambda(n)\subseteq K$
for all $\lambda\in B_\Lambda(\lambda_0,\delta)$,
Lemma~\ref{jointcts} guarantees that $S_\lambda(t,n)B$
is continuous in $(\lambda,B)\in \Lambda\times {\it CB}(K)$.
Viewing \eqref{inter} as a composition of continuous functions
now yields the continuity of $\A_\lambda(t)$ at $\lambda\in \Lambda_*$.
\end{proof}

\section{Residual continuity of uniform attractors}\label{uniformsec}

This section develops the theory of residual continuity
of uniform attractors with respect to a parameter.
A key component of the proof is the expression for the uniform attractor
as a union of the uniform $\omega$-limit sets given by
\begin{equation}\label{AAis}
	\AA_\lambda=\overline{\bigcup_{n\in\N}\Omega_\lambda(B_X(0,n))},
\end{equation}
where, as in Chapter VII of \cite{CV},
we define
$$
	\Omega_\lambda(B)=\bigcap_{\tau\in\R}\ \overline{\bigcup_{s\in\R}\
	\bigcup_{t\ge\tau}S_\lambda(t+s,s)B}
\qquad\hbox{for any set}\qquad B\subseteq X.
$$

We can now state our main result on uniform attractors.

\begin{theorem}\label{thm31}
Suppose that there exists a compact set $K\subseteq X$ such that
\begin{itemize}
\item[\rm (a)] for every bounded $B\subseteq X$ and
	each $\lambda\in\Lambda$ there exists a $t_{B,\lambda}$ such that
\begin{equation}\label{U}
S_\lambda(t+s,s)B\subseteq K\qquad\mbox{for all}\qquad
t\ge t_{B,\lambda}\quad\mbox{and}\quad s\in\R;
\end{equation}
and
\item[\rm (b)]
for any $t>0$ the mapping $S_\lambda(t+ s,s)x$ is continuous
in $\lambda\in\Lambda$ uniformly for $s\in\R$ and $x\in K$.
\end{itemize}
Then the uniform attractor
$\AA_\lambda$ is continuous in $\lambda$ at a residual subset of $\Lambda$.
\end{theorem}

In the preceding theorem,
assumption
(a) is sufficient
for the existence of a uniform attractor
given by the uniform $\omega$-limit (\ref{AAis})
with $\lambda$ fixed,
and
assumption (b) provides some uniform continuity of the
processes $S_\lambda$ in a way that depends only on the elapsed time.
More specifically,
given $\lambda_0\in\Lambda$, $t>0$ and $\epsilon>0$, there is $\delta>0$
depending only on $\lambda_0$, $t$ and $\epsilon$
such that for any $\lambda\in \Lambda$
with $d_\Lambda(\lambda,\lambda_0)<\delta$
$$d_X\big(S_\lambda(t+ s,s)x,S_{\lambda_0}(t+ s,s)x\big) < \epsilon
\qquad\hbox{for all}\qquad
s\in\R\quad\hbox{and}\quad x\in K.$$

\begin{proof}
Take $\lambda\in\Lambda$.  Applying \eqref{unicv} to $S=S_\lambda$ gives
$$
\lim_{t\to\infty}\sup_{s\in\R}
	\rho_X\big(S_\lambda(t+s,s)B,\AA_\lambda\big)=0
$$
for any bounded $B\subseteq X$. Since
$$\rho_X\bigg(\overline{\bigcup_{s\in\R}S_\lambda(t+s,s)B},\AA_\lambda\bigg)
	=\rho_X\bigg(\bigcup_{s\in\R}S_\lambda(t+s,s)B,\AA_\lambda\bigg)
	\le \sup_{s\in\R} \rho_X \big(S_\lambda(t+s,s)B,\AA_\lambda\big),$$
we obtain
\begin{equation}\label{unifcgve}
\lim_{t\to\infty}
	\rho_X\bigg(\overline{\bigcup_{s\in\R}S_\lambda(t+s,s)B},
		\AA_\lambda\bigg)=0.
\end{equation}

Let $T_n=\tau_{B,\lambda}$
for $n\in\N$
be given by (\ref{U})
where $B=B_X(0,n)$.
Then
\begin{align*}
&\Omega_\lambda(B_X(0,n))
	\subseteq\bigcap_{\tau\ge T_n}\
	\overline{\bigcup_{s\in\R}\ \bigcup_{t\ge\tau}S_\lambda(t+s,s)B_X(0,n)}\\
  &=\bigcap_{\tau\ge T_n}\
	\overline{\bigcup_{s\in\R}\
	\bigcup_{t\ge\tau}S_\lambda(t+s,T_n+s)S_\lambda(T_n+s,s)B_X(0,n)}\\
  &\subseteq\bigcap_{\tau\ge T_n}\ \overline{\bigcup_{s\in\R}\
	\bigcup_{t\ge\tau}S_\lambda(t+s,T_n+s)K}
  =\bigcap_{\tau\ge T_n}\
	\overline{\bigcup_{\eta\in\R}\
	\bigcup_{t\ge\tau}S_\lambda(t-T_n+\eta,\eta)K}\\
  &=\bigcap_{\tau\ge T_n}\ \overline{\bigcup_{\eta\in\R}\
	\bigcup_{s\ge \tau-T_n}S_\lambda(s+\eta,\eta)K}
  =\bigcap_{t\ge0}\ \overline{\bigcup_{\eta\in\R}\
	\bigcup_{s\ge t}S_\lambda(s+\eta,\eta)K}.
\end{align*}
It follows from (\ref{AAis}) that
\begin{equation}\label{x1}
	\AA_\lambda\subseteq\bigcap_{t\ge0}\
	\overline{\bigcup_{\eta\in\R}\
		\bigcup_{s\ge t}S_\lambda(s+\eta,\eta)K}.
\end{equation}

Let $T=\tau_{K,\lambda}$ in \eqref{U}.
By \eqref{x1} we have for $t\ge 0$ that
\begin{align*}
\AA_\lambda
&\subseteq \overline{\bigcup_{\eta\in\R}\ \bigcup_{s\ge t+T}
	S_\lambda(s+\eta,\eta)K} \\
&\subseteq \overline{\bigcup_{\eta\in\R}\ \bigcup_{s\ge t+T}
	S_\lambda\big(t+(s-t+\eta),s-t+\eta\big) S_\lambda(s-t+\eta,\eta)K} \\
&\subseteq \overline{\bigcup_{\eta\in\R}\ \bigcup_{s\ge t+T}
	S_\lambda\big(t+(s-t+\eta),s-t+\eta\big) K}.
\end{align*}
Thus
\begin{equation}\label{x2}
\AA_\lambda
\subseteq \overline{\bigcup_{z\in\R} S_\lambda(t+z,z) K}.
\end{equation}
Taking $B=K$ in \eqref{unifcgve} yields
\begin{equation}\label{basiclim}
\Delta_X\bigg(\overline{\bigcup_{s\in\R}S_\lambda(t+s,s)K},\AA_\lambda\bigg)
	\to0\qquad\mbox{as}\qquad t\to\infty.
\end{equation}

Define $K_\lambda(t)=\bigcup_{s\in\R}S_\lambda(t+s,s)K$,
fix $\lambda_0\in \Lambda$ and let $t>0$.
Given $\epsilon >0$ choose $\delta>0$ as in (b).
Then for any $x\in K$, we have
$$ d_X\big(S_\lambda(t+ s,s)x,S_{\lambda_0}(t+ s,s)x\big)
	< \epsilon\qquad\text{for any}\qquad s\in\R. $$
Hence
$$ \rho_X\big(K_\lambda(t),K_{\lambda_0}(t)\big)\le \epsilon
	\qquad\mbox{and}\qquad
	\rho_X\big(K_{\lambda_0}(t),K_{\lambda}(t)\big)\le \epsilon$$
imply that
$$ \Delta_X\big(\overline{K_\lambda(t)},\overline{K_{\lambda_0}(t)}\big)
	\le \epsilon.$$
Consequently
$$
	\lambda\mapsto \overline{K_\lambda(t)}
	=\overline{\bigcup_{s\in\R}S_\lambda(t+s,s)K}
$$
is continuous from $\Lambda$ into ${\it BC}(X)$.
The result now follows from \eqref{basiclim} and Theorem~\ref{BCT}.
\end{proof}

\section{Continuity everywhere and equi-attraction}\label{equi}

In this section, we extend the notion of equi-attraction
and the results on continuity of global attractors
with respect to a parameter $\Lambda$ in \cite{KL}
to non-autonomous systems.
We first show that that the continuity of pullback attractors
with respect to a parameter $\lambda\in\Lambda$ is equivalent
to pullback equi-attraction when $\Lambda$ is compact.
Next, we prove similar
results for uniform equi-attraction and uniform attractors.
Our methods are based on those used in Section 4 of \cite{HOR1}.

Assume~(L1) and~(L2) throughout this section where
$D$ is the set specified in~(L2).
Now
consider the following conditions:
\begin{enumerate}
\item[(U1)]
Pullback equi-dissipativity at time $t\in \R$: there exists $s_0\le t$ and a bounded set $B$ such that
$$ S_\lambda(t,s)D\subseteq B
\qquad\hbox{for every}\qquad
	s\le s_0\quad\hbox{and}\quad \lambda\in \Lambda.$$
\item[(U2)] Pullback equi-attraction at time $t\in \R$:
$$ \lim_{s\to -\infty} \sup_{\lambda\in \Lambda}
	\rho_X \big(S_\lambda(t,s)D,\A_\lambda(t)\big)=0.$$
	
\item[(U3)] There is a bounded set $D_1$ and a function $s_*(t)$ such
that $s_*(t)\le t$ and
\begin{equation}\label{SD1}
S_\lambda(t,s)D_1\subseteq D_1
\qquad\hbox{for every}\qquad
	s\le s_*(t)\quad \hbox{and}\quad\lambda\in \Lambda.
\end{equation}
\end{enumerate}

We remark that (U3) is the uniform version of (U1)
commonly obtained while proving the existence of pullback attractors.
In the autonomous case (U3) is identical to condition
(4.5) in \cite{HOR1}.
Our analysis here relies on the following version of Dini's theorem, which also appears in \cite{HOR1}.

\begin{lemma}[Theorem 4.1 in \cite{HOR1}]\label{Dini}
Let $K$ be a compact metric space and $Y$ be a metric space.
For each $n\in \N$, let $f_n\colon K\mapsto Y$ be a continuous map.
Assume $f_n$ converges to a continuous function
$f\colon K\mapsto Y$ as $n\to\infty$ in the following monotonic way
$$
d_Y\big(f_{n+1}(x),f(x)\big)\le d_Y\big(f_{n}(x),f(x)\big)
	\qquad\hbox{for all }
n\in\N\quad\hbox{ and for every } x\in K.
$$
Then $f_n$ converges to $f$ uniformly on $K$ as $s\to\infty$.
\end{lemma}

First, we deal with the pullback attractors.

\begin{theorem}\label{equiback}
Fix $t\in \R$. If {\rm (U1)} and {\rm (U2)} hold
then $\A_\lambda(t)$ is continuous at every $\lambda\in \Lambda$.
Conversely, suppose that  $\Lambda$ is a compact metric space and that
{\rm (U3)} holds; then if $\lambda\mapsto \A_\lambda(t)$
is continuous on $\Lambda$
then {\rm (U2)} holds with $D$ replaced with $D_1$.
\end{theorem}

\begin{proof}
Let $t\in \R$ be such that (U1) and (U2) hold.
Following the same arguments used to obtain
\eqref{SA0} and \eqref{Alan} in
the proof of Theorem~\ref{thm23},
except with $t$ replacing $n$,
we have for each $\lambda\in \Lambda$ and $s\le t$ that
\begin{equation}\label{uc0}
 \A_\lambda(t)\subseteq S_\lambda(t,s)D  \end{equation}
  and
  \begin{equation}\label{uc1}
\A_\lambda(t)=\lim_{s\to-\infty}\overline{S_\lambda(t,s)D}.
  \end{equation}
By (U2) the convergence in \eqref{uc1} is uniform
in $\lambda\in \Lambda$ as $s\to -\infty$.
Moreover, Lemma~\ref{Dncts} implies for $s\le t$ that the function
$\lambda\mapsto \overline{S_\lambda(t,s)D}$ is continuous in $\lambda$
on $\Lambda$.
Thus, the limit function $\lambda\mapsto\A_\lambda(t)$
is continuous on $\Lambda$.

We now prove the converse.
Assume (U3) and that
$\lambda\mapsto \A_\lambda(t)$ is continuous on $\Lambda$
for some $t\in \R$.
Let $s_0=s_*(t)-1$, $s_1=s_*(s_0)-1$ and  $s_{n+1}=s_*(s_n)-1$ for $n\ge 1$.
Then the sequence $\{s_n\}_{n=0}^\infty$ is strictly decreasing and
$s_n\to -\infty$ as $n\to\infty$.
By (P2) and (U3) we have
\begin{equation}\label{incl5}
	S(t,s_{n+1})D_1=S(t,s_n)S(s_n,s_{n+1})D_1\subseteq S(t,s_n)D_1.
\end{equation}

Replacing $D$ by $D_1$ in \eqref{uc0} and \eqref{uc1} yields
 \begin{equation}\label{incl6}
\A_\lambda(t)\subseteq S_\lambda(t,s)D_1 \qquad\hbox{for all}\qquad
	s\le t
 \end{equation}
and
\begin{equation}\label{limit1}
 \A_\lambda(t)=\lim_{s\to-\infty} \overline{S_\lambda(t,s)D_1}.
\end{equation}
We infer from \eqref{incl5} and \eqref{incl6} that
\begin{equation}\label{mono2}
\Delta_X\big(\overline{S(t,s_{n+1})D_1},\A_\lambda(t)\big)
	\le \Delta_X\big(\overline{S(t,s_{n})D_1},\A_\lambda(t)\big).
\end{equation}
Therefore,
the convergence given in \eqref{limit1}
is monotonic
along the sequence $s=s_n$, and
consequently,
Lemma~\ref{Dini} implies
$\overline{S_\lambda(t,s_n)D_1}$ converges to
$\A_\lambda(t)$ as $n\to\infty$ uniformly in $\lambda\in\Lambda$.
Thus,
\begin{equation}\label{limit2}
 \lim_{n\to \infty} \sup_{\lambda\in \Lambda}
	\rho_X\big(\overline{S_\lambda(t,s_n)D_1},\A_\lambda(t)\big)=0.
\end{equation}

To obtain \eqref{limit2} in the continuous limit as $s\to -\infty$
suppose $s\in(s_{n+2},s_{n+1})$.
Then by definition $s<s_*(s_n)$ so that by (U3) we obtain
  $$\A_\lambda(t)\subseteq S(t,s)D_1=S(t,s_n)S(s_n,s)D_1
	\subseteq S(t,s_n)D_1.$$
Hence,
$$ \rho_X\big(\overline{S_\lambda(t,s)D_1},\A_\lambda(t)\big)
	\le \rho_X\big(\overline{S_\lambda(t,s_n)D_1},\A_\lambda(t)\big)
	\quad\hbox{for every}\quad \lambda\in\Lambda.$$
This and \eqref{limit2} prove
\begin{equation*}
	\lim_{s\to - \infty} \sup_{\lambda\in \Lambda}
	\rho_X\big(\overline{S_\lambda(t,s)D_1},\A_\lambda(t)\big)=0,
\end{equation*}
which is exactly (U2) with $D$ replaced by $D_1$.
\end{proof}

For uniform attractors we work under the standing assumption that
there exists a set~$K$ such that (a) of Theorem~\ref{thm31} holds.
We say that $\AA_\lambda$ is uniformly equi-attracting if
\begin{equation}
\lim_{t\to\infty} \sup_{\lambda\in \Lambda,s\in\R}
	\rho_X\big(S_\lambda(t+s,s)K,\AA_\lambda\big)=0.
\label{ulim}
\end{equation}
In our analysis we further consider the case where any trajectory
starting in $K$ uniformly re-enters $K$ within a certain time $T_0$.
This is characterized by the following condition.
\begin{enumerate}
\item[(U4)] Assume there  exists $T_0\ge 0$ such that
\begin{equation}\label{ulam}
S_\lambda(t+s,s)K \subseteq K
\qquad\mbox{for all}\qquad
t\ge T_0,\quad\lambda\in\Lambda\quad\mbox{and}\quad s\in\R.
\end{equation}
\end{enumerate}
We are now ready to prove our main result on uniform equi-attraction.

\begin{theorem}\label{equiuni}
If $\AA_\lambda$ is uniformly equi-attracting,
then $\AA_\lambda$ is continuous on $\Lambda$.
Conversely, if $\AA_\lambda$ is continuous,
$\Lambda$ is compact and
{\rm (U4)} is satisfied, then
$\AA_\lambda$ is uniformly equi-attracting.
\end{theorem}

\begin{proof}
Under our standing assumption about the existence of $K$,
we have that \eqref{unifcgve}, \eqref{x2} and \eqref{basiclim} hold.
To prove that $\AA_\lambda$ is continuous, recall from
Theorem~\ref{thm31} that
$$\lambda\mapsto \overline{\bigcup_{s\in\R}S_\lambda(t+s,s)K}$$
is continuous.
By \eqref{x2} and \eqref{ulim} the limit
\eqref{basiclim} is uniform in $\lambda$.  In other words,
$$\overline{\bigcup_{s\in\R}S_\lambda(t+s,s)K}
\to\AA_\lambda
\quad\text{as}\quad t\to\infty
\quad\text{uniformly for}\quad\lambda\in\Lambda	.$$
Therefore, $\AA_\lambda$ is continuous at every $\lambda\in\Lambda$.

Conversely, let $T_0$ be in (U4) and set
$T_*=T_0+1\ge 1$. Let $t_n=nT_*$ for all $n\in\N$. Then $t_n\to\infty$ as $n\to\infty$.
For $s\in \R$,
\begin{align*}
 S_\lambda(t_{n+1}+s,s)K
&=S_\lambda(t_n+T_*+s,s)K=S_\lambda(t_n+T_*+s,T_*+s)S_\lambda(T_*+s,s)K\\
&\subseteq S_\lambda(t_n+T_*+s,T_*+s)K\subseteq \bigcup_{s\in\R} S_\lambda(t_n+s,s)K.
\end{align*}
Thus,
\begin{equation} \label{x0}
\overline{\bigcup_{s\in\R} S_\lambda(t_{n+1}+s,s)K}
	\subseteq \overline{\bigcup_{s\in\R} S_\lambda(t_n+s,s)K}.
\end{equation}
The inclusion \eqref{x2} then yields
\begin{equation*}
\Delta_X\left( \overline{\bigcup_{s\in\R} S_\lambda(t_{n+1}+s,s)K},\AA_\lambda\right)\le \Delta_X\left(\overline{\bigcup_{s\in\R} S_\lambda(t_n+s,s)K},\AA_\lambda\right).
\end{equation*}
which is the monotonicity needed for Lemma~\ref{Dini}.
It follows that
the convergence in \eqref{basiclim} taken along the sequence $t=t_n$
is uniform in $\lambda$.  In other words, that
  \begin{equation}\label{basiclim20}
\lim_{n\to\infty} \sup_{\lambda\in\Lambda }\Delta_X
	\bigg(\overline{\bigcup_{s\in\R}S_\lambda(t_n+s,s)K},\AA_\lambda\bigg)=0.
  \end{equation}

To obtain uniformity in the continuous limit as $t\to\infty$ suppose
$t\in(t_{n+1},t_{n+2})$.   Then $t-t_n>T_0$ and
\begin{align*}
S_\lambda(t+s,s) K
&= S_\lambda(t+s,t-t_n+s) S_\lambda(t-t_n+s,s) K\\
&\subseteq S_\lambda(t+s,t-t_n+s)K
=S_\lambda(t_n+z,z)K
\end{align*}
with $z=t-t_n+s$. Thus,
$$
S_\lambda(t+s,s) K \subseteq \bigcup_{z\in\R}S_\lambda(t_n+z,z)K.
$$
Together with \eqref{x2} we obtain that
$$
\AA_\lambda
\subseteq \overline{\bigcup_{s\in\R} S_\lambda(t+s,s) K}\subseteq
\overline{\bigcup_{s\in\R}S_\lambda(t_n+s,s)K}
$$
and therefore
\begin{equation}\label{mono3}
\Delta_X\bigg(\overline{\bigcup_{s\in\R}S_\lambda(t+s,s)K},\AA_\lambda\bigg)
\le
\Delta_X\bigg(\overline{\bigcup_{s\in\R}S_\lambda(t_n+s,s)K},\AA_\lambda\bigg).
\end{equation}

Combining \eqref{mono3} with \eqref{basiclim20} yields
\begin{equation*}
	\lim_{t\to\infty} \sup_{\lambda\in\Lambda}
	\Delta_X\bigg(\overline{\bigcup_{s\in\R}S_\lambda(t+s,s)K},
		\AA_\lambda\bigg)=0
\end{equation*}
and consequently
\begin{equation}\label{basiclim4}
    \lim_{t\to\infty} \sup_{\lambda\in\Lambda} \rho_X\bigg(\bigcup_{s\in\R}
    S_\lambda(t+s,s)K,\AA_\lambda\bigg)=0.
\end{equation}
This implies that $\AA_\lambda$ is uniformly equi-attracting.
\end{proof}

We make the following four remarks.
First, that our results for uniform attractors in Theorem \ref{equiuni} have not been established before in literature.
Second, our result in Theorem \ref{equiback} on everywhere continuity implying the equi-attraction is simpler and less technical than the similar ones in \cite{KP,CLR2009}.
Third, our result requires certain boundedness,  but not any extra compactness for the family of the processes $S_\lambda(\cdot,\cdot)$.
Finally, there are other papers (e.g.\  \cite{KLb}) that
establish the equivalence of equi-attraction and everywhere continuity
for a nonautonomous dynamical system $(\theta,\phi)$ with a cocycle
mapping $\phi$ on $X$ driven by an autonomous dynamical system $\theta$
acting on a base or parameter space $P$. Such considerations are
notably different from ours.

\section{Applications}\label{apps}

In this section we demonstrate the applicability of the abstract
theory developed in Sections~\ref{pullbacksec} and \ref{uniformsec}
to some well-known systems of ordinary and partial differential equations.
We also present some natural examples that are relevant for the
autonomous theory developed in \cite{HOR1} (see also \cite{BPY}).
We begin with the following simple observation about residual
sets.
The proof is elementary but included for the sake of clarity
and completeness.

\begin{lemma}\label{reslem}
  Let $(X,d)$ be a metric space, and suppose that
  \begin{equation}\label{X}
  X=\bigcup_{j=1}^\infty X_j.
  \end{equation}
If $Y_j\subseteq X_j$ is residual in $(X_j,d)$ for all $j\ge 1$,
then the set
  \begin{equation}\label{Y}
  Y=\bigcup_{j=1}^\infty Y_j
  \end{equation}
  is residual in $X$.
\end{lemma}

\begin{proof}
First, observe that if $X'\subseteq X$ and $Z\subseteq X'$ is
nowhere dense in $X'$ then $Z$ is nowhere dense in $X$.
Second, observe that if $A$ is residual in $X$ and $A\subseteq B$,
then $B$ is residual in $X$.
By hypothesis
$X_j\setminus Y_j={\textstyle\bigcup_{i=1}^\infty} A_{ji}$
where each $A_{ji}$ is nowhere dense in $X_j$.
It follows that
$$X\setminus Y=\bigcup_{j=1}^\infty (X_j\setminus Y)
	=\bigcup_{j=1}^\infty \bigcap_{k=1}^\infty  (X_j\setminus Y_k)
\subseteq \bigcup_{j=1}^\infty   (X_j\setminus Y_j)
	=\bigcup_{j=1}^\infty \bigcup_{i=1}^\infty A_{ji}.$$
By the first observation each $A_{ji}$ is nowhere dense in $X$.
Since
$
	X\setminus
		\bigcup_{j=1}^\infty \bigcup_{i=1}^\infty A_{ji}
	\subseteq Y,
$
the second observation implies that $Y$ is residual in $X$.
\end{proof}

\subsection{The Lorenz system}

The first application of our theory concerns the system of three
ordinary differential equations
introduced by Lorenz in \cite{Lorenz}.
Namely, we consider
\begin{equation}\label{Loz}
\left \{
\begin{aligned}
 x'&=-\sigma x+\sigma y,\\
 y'&=rx-y-xz,\\
 z'&=-bz+xy,
\end{aligned}\right.
\end{equation}
where $\sigma$, $b$ and $r$ are positive constants.
These equations have been widely studied as a model
of deterministic nonperiodic flow.
The standard bifurcation parameter of the Lorenz equations is $r$,
see \cite{CS}, but we will consider continuity of the
global attractor of the autonomous system with respect to
the full parameter set $\lambda=(\sigma,b,r)$.
Since physical measurements and numerical computations
in general employ only approximate values,
then considering perturbations in all three parameters
makes sense from a mathematical point of view.
As an example, we point out that Tucker \cite{Tucker} considered
an open neighborhood of the standard choice of parameters
$\lambda=(10,8/3,28)$
in his work on
the Lorenz equations.

As shown in Doering and Gibbon \cite{DG},
see also Temam \cite{Temam},
for any $\lambda\in(0,\infty)^3$ the solutions to (\ref{Loz}) generate a
semigroup $S_\lambda(t)$ for which there exists a corresponding global
attractor $\A_\lambda$.
Therefore, the requirement (G1) of Theorem~\ref{auto} is met.
The estimates in \cite{DG,Temam} also show that for
any compact subset $\Pi$ of $(0,\infty)^3$ that there is a
a bounded set $D$ such that given any bounded set $B\in\R^3$
there is $T>0$ such that
$$S_\lambda(t)B\subseteq D\qquad\hbox{for all}\qquad
	t\ge T\quad\hbox{and}\quad \lambda\in\Pi.$$
This guarantees that (G2) holds.
Assumption (G3), the continuity of $S_\lambda(t)$
with respect to $\lambda$, can be verified by considering
the equation for the difference of two solutions with different values
of the parameters and using a standard Gronwall-type argument.
Thus, the conditions of Theorem~\ref{auto} are satisfied, as a consequence of which we have the following result.

\begin{theorem}\label{Loz1}
  There is a residual and dense subset $\Lambda_*$ in $(0,\infty)^3$
  such that the function from $(0,\infty)^3\to{\it CB}(\R^3)$ given by
  $\lambda\mapsto\A_\lambda$ is continuous at
  every $\lambda\in\Lambda_*$.
\end{theorem}

\begin{proof}
For each $n\in\N$ let $\Lambda_n=[n^{-1},n]^3$ and define
$\Phi_n\colon \Lambda_n\mapsto{\it CB}(\R^3)$
by $\Phi_n(\lambda)=\A_\lambda$.
Theorem~\ref{auto} implies
there is a residual set $\Lambda_{*,n}$
in $[n^{-1},n]^3$ such that $\Phi_n$
is continuous at each point in $\Lambda_{*,n}$ with
respect to the Hausdorff metric.
Set
$$
\Lambda_*=\bigcup_{n=2}^\infty \big(\Lambda_{*,n}\cap \Lambda_n^\circ\big)
\qquad\hbox{where}\qquad
\Lambda_n^\circ=(n^{-1},n)^3.
$$
Since, the function
$\Phi\colon (0,\infty)^3\to{\it CB}(\R^3)$
defined by
$\Phi(\lambda)=\A_\lambda$
is continuous at each point in
$\Lambda_{*,n}\cap \Lambda_n^\circ$,
then $\Phi$ is continuous
at each point in $\Lambda_*$.
Since $\Lambda_{*,n}\cap \Lambda_n^\circ$ is residual and dense
in $\Lambda_n^\circ$, then
Lemma~\ref{reslem} implies that
$\Lambda_*$ is a residual subset of $(0,\infty)^3$.
Moreover, since each $\Lambda_{*,n}\cap\Lambda_n^\circ$
is dense in $(n^{-1},n)^3$,
then $\Lambda_*$ is dense in
$(0,\infty)^3$.\end{proof}


As a simple illustration of the non-autonomous theory,
let
$r(t)$ be a fixed   $C^1$-function on $\R$ and $R_0$ a constant such
that
\begin{equation}\label{rcond}
|r(t)|, |r'(t)| \le R_0
	\qquad\mbox{for all}\qquad t\in\R.
\end{equation}
Consider
the family of systems of ordinary differential equations
given by
\begin{equation}\label{nasLoz2}
\left \{
\begin{aligned}
 x'&=-\sigma x+\sigma y,\\
 y'&=r(t)x-y-xz,\\
 z'&=-b z+xy
\end{aligned}\right.
\end{equation}
indexed by the parameter $\lambda=(\sigma,b)$.

Note that the model \eqref{nasLoz2} and assumption \eqref{rcond} are relevant in some climate models, see for example \cite{DS}. In particular, the function $r(t)$ can be a finite sum of  sinusoidal functions.
Making the standard change of variable   $w=z-\sigma -r(t)$ we rewrite
\eqref{nasLoz2} as
\begin{equation}\label{SLt}
\left \{
\begin{aligned}
 x'&=-\sigma x+\sigma y,\\
 y'&=-y-\sigma x-xw,\\
 w'&=-bw+xy+F(t)\end{aligned}
\right.
\end{equation}
with
$$F(t)=-b(\sigma +r(t))-r'(t).
$$
Thanks to condition \eqref{rcond} setting
$F_0=b(\sigma +R_0)+R_0$ yields
\begin{equation}\label{F0}
|F(t)|\le F_0
\qquad\hbox{for all}\qquad t\in\R.
\end{equation}

Similar estimates to those in
\cite{DG} and \cite{Temam}, based on the formulation \eqref{SLt},
show for each $(\sigma,b)\in (0,\infty)^2$ that
the system \eqref{nasLoz2}
generates a process $S_{\sigma,b}(\cdot,\cdot)$, that there exists
pullback attractors $\A_{\sigma,b}(t)$ for every $t\in\R$, and that
the uniform attractor $\AA_{\sigma,b}$ exists.
For the sake of completeness we present explicit estimates here,
which will also be used in the next theorem.

Let $v(t)=(x(t),y(t),z(t))$ be a solution of \eqref{Loz}, and $u(t)=(x(t),y(t),w(t))$.
 Note that $$|v|\le |u|+\sigma +R_0\quad \text{and}\quad|u|\le |v|+\sigma +R_0.$$
We have from \eqref{SLt}, \eqref{F0}
and by Cauchy's inequality that
\begin{equation*}
\frac12 \frac{\d}{\d t}(x^2+y^2+w^2)+\sigma x^2+y^2+bw^2=Fw \le \frac{b}2 w^2+\frac{F_0^2}{2b}.
\end{equation*}
Upon setting $\sigma_0=\min\{1,\sigma ,b/2\}$ it follows that
\begin{equation*}
\frac{\d}{\d t}|u|^2+2\sigma_0 |u|^2\le \frac{F_0^2}{b}.
\end{equation*}
This implies for all $t\ge 0$ that
\begin{equation*}
 |u(t)|^2\le |u(0)|^2 e^{-2\sigma_0 t}+\frac{F_0^2}{2\sigma_0 b};
\end{equation*}
hence
\begin{equation}\label{ut}
 |u(t)|\le |u(0)| e^{-\sigma_0 t}+\frac{F_0}{\sqrt{2\sigma_0  b}}\le (|v(0)|+\sigma +R_0) e^{-\sigma_0 t}+\frac{F_0}{\sqrt{2\sigma_0 b}}.
\end{equation}
Thus,
\begin{equation}\label{vt}
 |v(t)|\le (|v(0)|+\sigma +R_0) e^{-\sigma_0 t}+\frac{F_0}{\sqrt{2\sigma_0 b}}+ (\sigma +R_0).
\end{equation}
By
\eqref{ut} and \eqref{vt},  we have for all $t\ge 0$ that
\begin{equation}\label{uv}
|u(t)|, |v(t)|\le R_1
\quad\hbox{where}\quad
R_1=|v(0)|+2(\sigma +R_0) + \frac{F_0}{\sqrt{2\sigma_0 b}}.
\end{equation}

Next we consider the continuity in $\lambda=(\sigma,b)$.
For $i\in \{ 1,2\}$ let $\lambda_i=(\sigma_i,b_i)\in(0,\infty)^2$ be given
and let $v_i(t)=(x_i(t),y_i(t),z_i(t))$ be the corresponding solution of \eqref{Loz}.  Define
$$
	u_i(t)=(x_i(t),y_i(t),w_i(t))=(x_i(t),y_i(t),z_i(t)-\sigma _i-r(t)).
$$
Let $\bar\lambda
=(\bar \sigma , \bar b)
=\lambda_1-\lambda_2$ and
$\bar u=
(\bar x,\bar y,\bar w)=
u_1-u_2$.
Then \eqref{SLt} implies
\begin{equation*}
\left\{
\begin{aligned}
 {\bar x}'&=-\bar \sigma x_1-\sigma _2 \bar x + \bar \sigma y_1+\sigma _2 \bar y,\\
 {\bar y}'&=-\bar y-\bar x w_1-x_2 \bar w -\bar \sigma  x_1-\sigma _2 \bar x,\\
 {\bar w}'&=-\bar b w_1-b_2 \bar w +\bar xy_1+x_2 \bar y- \bar b(\sigma _1+r(t))-b_2 \bar \sigma .
\end{aligned}\right.
\end{equation*}
It follows that
\begin{align}
 \frac{\d}{\d t}|\bar u|^2 +\sigma _2\bar x^2+\bar y^2+b_2 \bar w^2&=-\bar \sigma (x_1+y_1)\bar x
 -\bar x w_1\bar y-\bar \sigma  x_1 \bar y \notag\\
&\quad -\bar b w_1\bar w+\bar x y_1\bar w -[\bar b(\sigma _1+r(t))+b_2\bar \sigma]\bar w. \label{dubar}
\end{align}

By neglecting $\sigma _2\bar x^2+\bar y^2+b_2 \bar w^2$ on the left-hand side of \eqref{dubar} and using estimate  \eqref{uv} to bound $|x_1|$, $|y_1|$, $|w_1|$ on the right-hand side, we obtain
\begin{align*}
 \frac{\d}{\d t}|\bar u|^2&\le 2R_1 |\bar \lambda| |\bar u|+R_1 |\bar u|^2 +|\bar \lambda| R_1 |\bar u|
 +|\bar \lambda|R_1 |\bar u|+R_1|\bar u|^2  + (|\lambda_1|+R_0+|\lambda_2|) |\bar \lambda||\bar u|\\
&= (4R_1+R_0+|\lambda_1|+|\lambda_2|)  |\bar \lambda| |\bar u|+2R_1 |\bar u|^2\\
&\le 2R_2|\bar u|^2+R_3^2|\bar \lambda |^2,
\end{align*}
where $R_1$ is defined in \eqref{uv} with $v=v_1$,
\begin{equation}\label{R23}
R_2=R_1+ \frac{1}8\quad \text{and}\quad R_3=R_0+4R_1+|\lambda_1|+|\lambda_2|.
\end{equation}
By Gronwall's inequality, we obtain  for $t\ge 0$ that
\begin{equation}\label{ubar}
|\bar u(t)|^2\le  |\bar u(0)|^2e^{2R_2 t} +  R_3^2 t e^{2R_2t} |\bar \lambda |^2.
\end{equation}

Let $\bar v=(\bar x,\bar y,\bar z)=v_1-v_2$.
Therefore, $\bar v=\bar u+(0,0, \sigma_1-\sigma_2)$.
It follows from \eqref{ubar} for $t\ge 0$ that
\begin{align*}\label{vbar}
|\bar v(t)|
&\le |\bar u(t)| +|\bar \lambda|
\le  |\bar u(0)|e^{R_2 t} +  R_3 \sqrt t e^{R_2t} |\bar \lambda |+|\bar\lambda|\\
&\le   (|\bar v(0)|+|\bar\lambda|)e^{R_2 t} +  R_3 \sqrt t e^{R_2t} |\bar \lambda |+|\bar\lambda|.
\end{align*}
Thus,
\begin{equation}\label{vbar}
|\bar v(t)|
\le  e^{R_2 t}\Big\{  |\bar v(0)|+(2+ R_3 \sqrt t )  |\bar \lambda |\Big\}\quad \hbox{for every}\quad t\ge 0.
\end{equation}

We are ready to obtain the continuity of
$\A_{\sigma,b}(t)$ and $\AA_{\sigma,b}$ as functions of $\sigma$ and $b$.

\begin{theorem}\label{Loz3}
  There is a residual and dense subset $\Lambda_*$ in $(0,\infty)^2$ such that
  the functions from $(0,\infty)^2\to{\it CB}(\R^3)$
  defined by
\begin{equation}\label{bothfun}
	(\sigma,b)\mapsto\A_{\sigma,b}(t)
\quad\hbox{for every}\quad t\in\R
\qquad\hbox{and}\qquad
  	(\sigma,b)\mapsto{\mathbb A}_{\sigma,b}
\end{equation}
  are continuous at every point $(\sigma,b)\in \Lambda_*$.
  \end{theorem}
\begin{proof}
Denote $\lambda=(\sigma,b)$.
Let $0<\delta\le\mu$.
Suppose $\lambda\in [\delta,\mu]^2$
and $|v(0)|\le M$.
Then
$$F_0\le F_*,\qquad \sigma_0\ge \sigma_*\qquad\hbox{and}\qquad
	b\ge 2\sigma_*$$
where
$F_*=\mu(\mu+R_0)+R_0$ and $\sigma_*=\min\{1,\delta/2\}$.
Consequently,
 \begin{equation*}
R_1\le R_{1,*}\qquad\hbox{where}\qquad
	R_{1,*}=M+2(\mu+R_0)+ \frac{F_*}{ 2 \sigma_*}.
\end{equation*}
Similarly denote
$\lambda_i=(\sigma_i,b_i)$.
Suppose $\lambda_i\in[\delta,\mu]^2$
and $|v_i(0)|\le M$ for $i=1,2$.
Then
\begin{equation*}
R_2\le R_{2,*}:=R_{1,*}+\frac{1}8\quad
	\text{and}\quad R_3\le R_{3,*}:=R_0+4R_{1,*}+2\sqrt2\mu.
\end{equation*}
Since $\R^3$ is finite dimensional, then every element of ${\it CB}(\R^3)$
is compact.
Therefore \eqref{vt} with $\Lambda=[\delta,\mu]^2$
may be used to verify requirement (L2$'$) and
condition~(a) of Theorem~\ref{thm31} while \eqref{vbar}
may be used to verify (L3$'$) and (b).

Let $\delta=1/n$ and $\mu=n$ for $n\ge 2$.
By Theorem~\ref{thm23} there is a residual set $\Lambda^{p}_{*,n}$ in $[1/n,n]^2$
such that the function $[1/n,n]^2\mapsto {\it CB}(\R^3)$ given by
$(\sigma,b)\mapsto\A_{\sigma,b}(t)$
is continuous
at each point in $\Lambda^p_{*,n}$ for all $t\in \mathbb R$.
Similarly, by Theorem~\ref{thm31} there is a residual
set $\Lambda^{u}_{*,n}$ in $[1/n,n]^2$ such that
$(\sigma,b)\mapsto{\mathbb A}_{\sigma,b}$ is continuous
at each point in $\Lambda^{u}_{*,n}$.
Let $\Lambda_{*,n}=\Lambda^{p}_{*,n}\cap \Lambda^{u}_{*,n} \cap (1/n,n)^2.$
Then $\Lambda_{*,n}$ is residual and dense in $(1/n,n)^2$ and
the maps given by \eqref{bothfun}
are continuous at every point $(\sigma,b)\in \Lambda_{*,n}$.
Set $
\Lambda_*=\bigcup_{n=2}^\infty \Lambda_{*,n}.
$
Since each $\Lambda_{*,n}$ is dense in $(1/n,n)^2$,
the set $\Lambda_*$ is dense in $(0,\infty)^2$.
Moreover, by Lemma~\ref{reslem}, $\Lambda_*$ is residual in $(0,\infty)^2$.
We finish noting that the functions defined in \eqref{bothfun} are
continuous at every point $(\sigma,\mu)\in\Lambda_*$.
\end{proof}

\subsection{The two-dimensional Navier-Stokes equations}\label{nsexample}
We now turn to the two-dimensional Navier--Stokes equations.
Let $\Omega$ be a bounded, open and connected set in $\R^2$ with
$C^2$ boundary (i.e.\ $\partial\Omega$ can be represented locally
as the graph of a $C^2$ function).
Consider the two-dimensional incompressible Navier--Stokes equations
in $\Omega$ with no-slip Dirichlet boundary conditions
\begin{equation}\label{NSE2D}
 \begin{cases}
  u_t-\nu \Delta u +(u\cdot \nabla) u=-\nabla p+f& \text{on } \Omega\\
 \nabla \cdot u=0 & \text{on } \Omega\\
u=0&\text{on }\partial \Omega,
 \end{cases}
\end{equation}
where $u=u(x,t)$ is the Eulerian velocity field,
$p=p(x,t)$ is the pressure, $\nu>0$
is the kinematic viscosity and $f=f(x,t)$ is the body force.

Define
$$
{\mathscr V}=\{\,v\in [C_c^\infty(\Omega)]^2:\nabla\cdot v=0\,\}
$$
and let $H$ and $V$ be the closures	
of ${\mathscr V}$ in the norms of
$[L^2(\Omega)]^2$ and $[H^1(\Omega)]^2$, respectively.
Note that $H$ is a Hilbert space with inner product
$(\cdot,\cdot)$
and corresponding norm $\|\cdot\|$
inherited from $[L^2(\Omega)]^2$.
Similarly $V$ is a Hilbert space, however, in this case we shall
use the norm $v\mapsto \|\nabla v\|$,
which
is equivalent to the $[H^1(\Omega)]^2$ norm
on $V$
due to the Poincar\'e inequality.
The Rellich--Kondrachov Theorem implies $V$ is compactly
embedded into $H$.  We denote the dual of $V$ by $V^*$
with the pairing $\langle u,v\rangle$ for
$u\in V^*$ and $v\in V$.

Following, for example \cite{TemamNSE},
we write \eqref{NSE2D} in functional form
as the equation
\begin{equation}\label{NSEfn}
 u_t + \nu Au +B(u,u)=P_L f
\end{equation}
in $V^*$
where $v(t)\in V$.
Here $P_L$ is the (Helmholtz--Leray)
orthogonal projection from $[L^2(\Omega)]^2$ onto $H$ and
$A$ and $B$ are the
continuous extensions of the operators
given by
$$Au=P_L(-\Delta u)
\qquad\hbox{and}\qquad
B(u,v)=P_L((u\cdot\nabla )v)
\qquad\hbox{for}\qquad
u,v\in{\mathscr V}$$
such that $A\colon V\mapsto V^*$ and $B\colon V\times V\mapsto V^*$.

Let $\lambda_1>0$ be the first eigenvalue of the Stokes operator.
With this notation
Poincar\'e's inequality may be written as
$\|\nabla v\|^2\ge \lambda_1\|v\|^2$ for all $v\in V$.
For convenience assume $f(t)\in H$ for all time $t\in\R$ so that
$P_L f= f$ in \eqref{NSEfn}.
When $f\in L^\infty(\R,H)$ further define the
Grashof number $G$ as
\begin{equation}\label{grashof}
	G
	=\frac{1}{\lambda_1\nu^2}\big\|f\big\|_{L^\infty(\R,H)}
\quad\hbox{where}\quad
	\big\|f\big\|_{L^\infty(\R,H)}=
	{\rm ess\,sup}\big\{\,\|f(t)\|:t\in\R\,\big\}.
\end{equation}
Note when $f$ is time independent, this definition
reduces to the definition of Grashof number given,
for example, in \cite{Temam}.

As shown in \cite{CLR} and references therein, when $f\in L^\infty(\R,H)$
the system \eqref{NSE2D} generates a
process $S_f(t,s)\colon H\mapsto H$
satisfying Definition \ref{processdef}
defined by $S_f(t,s)u_0=u(t)$
where $u(t)$ is the solution of \eqref{NSE2D} on $[s,\infty)$ with $u(s)=u_0$.
Moreover, a pullback attractor $\A_f(t)$ exists for every $t\in\R$
as does a uniform attractor $\AA_f$.
We therefore have (L1).

To obtain (L2$'$) and (L3$'$) we employ bounds
on individual solutions in terms of the Grashof number
similar to those which show the existence of absorbing sets in
$H$ and $V$ in the case when $f$ is time independent.
Such estimates may be found
in \cite{Temam} pages 109--111 and also
\cite{CFbook,Lady,JCR,SY02,TemamNSE} among others.
As they are simple we include the relevant calculations
in Appendix A.  

\begin{theorem}\label{apriori}
Let $f\in L^\infty(\R,H)$ and $G$ be defined as in \eqref{grashof}.
Suppose $u(t)=S_f(t,s)u_0$ where $u_0\in H$ with $\|u_0\|\le M$.  Then for all $t\ge s$
\begin{equation}\label{intV}
\nu\int_s^t\|\nabla u(\tau)\|^2\,\d\tau \le \|u_0\|^2 +
(t-s) \nu^3\lambda_1 G^2.
\end{equation}
Moreover, there exists a constant $t_0>0$,
depending only on $M$, $\nu$, and $\lambda_1$, such that $t-s\ge t_0$ implies that
\begin{equation}\label{Vbound}
	\|u(t)\|^2\le 2\nu^2 G^2
\qquad\hbox{and}\qquad
	\|\nabla u(t)\|^2\le \rho(G),
\end{equation}
where $\rho(G)$ is an increasing function of $G$ that also depends
on $\nu$ and $\lambda_1$.
\end{theorem}

Noting that $L^\infty(\R,H)$ is a complete metric space with
respect to the norm described in \eqref{grashof},
we are now ready to obtain the
continuity of $\A_f(t)$ and $\AA_f$ as functions of $f$.

\begin{theorem}\label{thm42}
There is a residual and dense subset $\Lambda_*$ in $L^{\infty}(\R,H)$
such that the maps from
$L^{\infty}(\R,H)\mapsto{\it CB}(H)$ given by
\begin{equation}\label{nsboth}
	f\mapsto\A_f(t)\quad\hbox{for every}\quad t\in\R
	\qquad\hbox{and}\qquad f\mapsto \AA_f
\end{equation}
are continuous at every point $f\in \Lambda_*$.
\end{theorem}

\begin{proof}
Given $n>0$ let
$$\Lambda_n=
	\big\{\, f\in L^\infty(\R,H) : G\le n\,\big\},
$$
where  $G$ is the Grashof number defined in \eqref{grashof},
and let $K$ be the ball of radius $\rho(n)$ in $V$.
We remark that
$\Lambda_n$ is a complete metric space and that
$K$ is a compact subset of $H$.  To obtain (L2$'$) it is
enough to show that $\A_f(t)\subset K$ for every $f\in\Lambda_n$ and
$t\in\R$.
In light of Theorems 11.3 and 2.12 of \cite{CLR} we recall that
\begin{equation*}
\A_\nu(t)=\overline{\bigcup\{\omega(B,t):\ B\text{ is a bounded set in } H\}},
\end{equation*}
where, according to  Definition 2.2 in \cite{CLR},
\begin{equation*}
\omega(B,t)=\bigcap_{\sigma\le t}\overline{\bigcup_{s\le \sigma} S_f(t,s)B}.
\end{equation*}

Now, given any bounded $B\subset H$, there is $M$ large enough such
that $u_0\in B$ implies $\|u_0\|\le M$.
From Theorem~\ref{apriori} there is $t_0$ large enough
such that
$$\|\nabla S_f(t,s)u_0\|\le \rho(G)\le \rho(n)
\qquad\hbox{whenever}\qquad t-s\ge t_0.$$
Therefore $S_f(t,s)u_0\in K$ and consequently
$$
	\omega(B,t) \subseteq \overline{\bigcup_{s\le t-t_0} S_f(t,s) B}
		\subseteq K.
$$
It follows that $\A_f(t)\subseteq K$.

To show that (L3$'$) holds, let
$f_1,f_2\in \Lambda_n$ and consider
the solutions
$$u_1(t)=S_{f_1}(t,s)u_0\qquad\hbox{and}\qquad
	u_2(t)=S_{f_2}(t,s)u_0$$
Then $w=u_1-u_2$ satisfies
\begin{equation}\label{diq2}
\frac{\d w}{\d t}+\nu Aw+B(u_1,w)+B(w,u_2)=f\quad\mbox{where}\quad f=f_1-f_2.
\end{equation}
The 2D Ladyzhenskaya inequality
$\|u\|_{L^4}\le c_L\|u\|^{1/2}\|\nabla u\|^{1/2}$ in conjunction with
the H\"older inequality implies that
\begin{equation}\label{Bineq}
 |\langle B(u,v),w\rangle |\le c_B \|u\|^{1/2}\|\nabla u\|^{1/2}\cdot \|\nabla v\|\cdot  \|w\|^{1/2} \|\nabla w\|^{1/2}
\end{equation}
where $c_B=c_L^2$.
Taking the inner product of \eqref{diq2} with $w$ and using \eqref{ortho} and \eqref{Bineq} gives
\begin{align*}
\frac12\frac{\d}{\d t}\|w\|^2+\nu \|\nabla w\|^2
& \le |\langle B(w,u_2),w\rangle|+   \|f\| \|w\| \\
& \le c_B \|w\| \|\nabla w\| \|\nabla u_2\| + \lambda_1^{-1/2} \|f\| \|\nabla w\| \\
&\le \frac\nu4\|\nabla w\|^2+ \frac{c_B^2 \|w\|^2 \|\nabla u_2\|^2}{\nu}+\frac\nu4\|\nabla w\|^2+\frac{\|f\|^2}{\nu\lambda_1}.
\end{align*}
Thus, we have
\begin{equation*}
\frac{\d}{\d t}\|w\|^2\le \frac{2c_B^2 \|\nabla u_2\|^2}{\nu}\|w\|^2 +\frac{2\|f\|^2}{\nu\lambda_1}.
\end{equation*}

Applying Gronwall's inequality we obtain
\begin{equation*}\label{wt}
\|w(t)\|^2\le \frac{2
\|f\|_{L^\infty(\R,H)}^2
}
		{\nu\lambda_1}
		\int_s^t
	\exp \Big( \frac{2c_B^2}{\nu}
	\int_\tau^t \|\nabla u_2(s)\|^2\,\d s \Big)\,\d\tau.
\end{equation*}
Using estimate \eqref{intV} for $u_2$ we have that
\begin{align*}
        \int_s^t
    \exp \Big( &\frac{2c_B^2}{\nu}
    \int_\tau^t \|\nabla u_2(s)\|^2\,\d s \Big)\,\d\tau\\
	&\le (t-s) \exp\bigg\{\frac{2c_B^2}{\nu^2}
	\Big(\|u_0\|^2+(t-s)\frac{\|f_2\|_{L^\infty(\R,H)}^2}
		{\nu\lambda_1}\Big)\bigg\}
	\le C(t,s)
\end{align*}
where
$$
	C(t,s)=(t-s) \exp\Big\{\frac{2c_B^2}{\nu^2}
	\big(M^2+(t-s){\nu^3\lambda_1 n^2}
		\big)\Big\}.
$$
It follows that
$$
\|S_{f_1}(t,s)u_0 - S_{f_2}(t,s)u_0\|^2
	\le \frac{2C(t,s)}{\nu\lambda_1}\|f_1-f_2\|_{L^\infty(\R,H)}^2.
$$
Therefore $S_f(t,s)u_0$ is continuous in $f$ uniformly
for $u_0$ in bounded subsets of $H$.
Thus, (L3) and consequently (L3$'$) holds.

By Theorem~\ref{thm23} there is a residual set $\Lambda^p_{*,n}$
in $\Lambda_n$ such that the maps from $\Lambda_n\to {\it CB}(H)$
defined by $f\mapsto\A_f(t)$ for every $t\in\R$ are continuous
at each point in $\Lambda^p_{*,n}$.
Note that the analysis which proves (L2$'$) and (L3$'$) also shows
that conditions (a) and (b) of Theorem~\ref{thm31} are satisfied.
Therefore, there is also a residual set $\Lambda^u_{*,n}$ such
that the map defined by $f\mapsto\AA_f$ is continuous
at each point in $\Lambda^u_{*,n}$.
Set $\Lambda_{*,n}=\Lambda^p_{*,n}\cap\Lambda^u_{*,n}$ and
$$\Lambda_*
	=\bigcup_{n=1}^\infty( \Lambda_{*,n}\cup\Lambda_n^\circ)
\qquad\hbox{where}\qquad
	\Lambda_n^\circ=
    \big\{\, f\in L^\infty(\R,H)
    : \|f\|_{L^\infty(\R,H)}<\nu^2\lambda_1 n\,\big\}.
$$
Arguments identical to those given at the end of the proof
for Theorem~\ref{Loz1} are now sufficient to finish this proof.
\end{proof}

Before closing, let us draw a few consequences from Theorem~\ref{thm42}.

\begin{corollary}
\label{nssubset}
Let $\Lambda\subseteq L^\infty(\R,H)$ be closed.
There is a residual and dense set $\Lambda_*$ in $\Lambda$
such that the maps from $\Lambda\mapsto{\it CB}(H)$ given by
\eqref{nsboth}
are continuous at every point in~$\Lambda_*$.
\end{corollary}

\begin{proof}
While it is not, in general, true that a set which is
residual in $L^\infty(\R,H)$ is necessarily residual in $\Lambda$,
we can argue as follows.
First observe that all estimates used in the proof of
Theorem~\ref{nssubset} also hold when considering a smaller
collection of forces.
Since a closed subset of a complete metric space is also
complete, Theorems~\ref{thm23} and \ref{thm31}
apply equally well to the sets $\Lambda_n$ given by
$$\Lambda_n=
	\big\{\, f\in \Lambda : \|f\|_{L^\infty(\R,H)}\le \nu^2\lambda_1 n \,\big\}.
$$
We therefore obtain a $\Lambda_*$ residual in $\Lambda$ that
satisfies the desired continuity conditions.
\end{proof}

Now consider the autonomous case in which the
two-dimensional incompressible Navier--Stokes
equations are forced by a time-independent function $f\in H$.
Although it is possible to apply
Theorem~\ref{auto} using the same analysis as before
to show that (G1), (G2) and (G3) hold,
we instead apply Corollary~\ref{nssubset} to obtain the following result.

\begin{corollary}
\label{autons}
There is a residual and dense set $\Lambda_*$ in $H$ such that
the map from $H\mapsto{\it CB}(H)$ given by
$f\mapsto\A_f$ is continuous at every point $f\in\Lambda_*$.
\end{corollary}
\begin{proof}
Since the set of time-independent forcing
functions may be viewed as a closed subset of $L^\infty(\R,H)$
then there is a residual set $\Lambda_*$ in $H$.
Since the global attractor $\A_f$ in the autonomous case is identical
with the pullback attractor $\A_f(t)$ for all $t\in\R$ when $f$ is
time independent (Lemma 1.19 in \cite{CLR}), we may immediately apply Corollary~\ref{nssubset} to
obtain the desired result.
\end{proof}

We close with an example in which we fix the forcing $f=f_0$
where $f_0\in L^\infty(\R,H)$ and consider the family of attractors
parameterized by viscosity $\nu$.

\begin{corollary}
\label{nsvisc}
There is a residual and dense set $\Lambda_*$ in $(0,\infty)$ such that
the maps from $(0,\infty)\mapsto{\it CB}(H)$ given by
\begin{equation}\label{viscmaps}
\nu\mapsto\A_\nu(t)\quad\hbox{for every}\quad t\in\R
\qquad\hbox{and}\qquad \nu\to\AA_\nu
\end{equation}
are continuous at every point $\nu\in\Lambda_*$.
\end{corollary}
\begin{proof}
The change of variables $v=\nu^{-1} u$ and $\tau=\nu t$ transforms
\eqref{NSEfn} into
\begin{equation}\label{novisc}
	\frac{\d v}{\d \tau} + A v + B(v,v)= \nu^{-2} f_0.
\end{equation}

For $g\in L^\infty(\R,H)$, denote by $\B_{g}(\tau)$ the pullback attractor of \eqref{novisc} with the right-hand side being replaced by $g$,
and by $\BB_{g}$ the uniform attractor.

Let $\Lambda'=\{\, c f_0 : c \in [0,\infty)\,\}.$
Since $\Lambda'$ is a closed subset of
$L^\infty(\R,H)$ then Corollary~\ref{nssubset}
implies there exists a residual (and so dense) set $\Lambda'_*$ in $\Lambda'$
such that the maps from $\Lambda'\mapsto {\it CB}(H)$ given by
\begin{equation}\label{gB}
	Q_1: g\mapsto\B_g(\tau)\quad\hbox{for every}\quad \tau\in\R
	\qquad\hbox{and}\qquad Q_2: g\mapsto \BB_g
\end{equation}
are continuous at each point in $\Lambda'_*$.

Define
$\Lambda_*=\{\, 1/\sqrt c : cf_0\in\Lambda'_*\hbox{ and }c>0\,\}$.
Since the mapping $\xi:(0,\infty)\to \Lambda'\setminus\{0\}$ given by
$\xi(\nu)=\nu^{-2} f_0$ is a continuous bijection, then $\Lambda_*=\xi^{-1}(\Lambda_*'\setminus\{0\})$ is
residual and dense in $(0,\infty)$.  Then, by \eqref{gB}, the maps from
$(0,\infty)\to{\it CB}(H)$ given by
\begin{equation}\label{Bcont}
 	Q_3= Q_1 \circ \xi:\nu\mapsto\B_{\xi(\nu)}(\tau)\quad\hbox{for every}\quad \tau\in\R
	\qquad\hbox{and}\qquad Q_4=P_2\circ \xi: \nu\mapsto \BB_{\xi(\nu)}
\end{equation}
are also continuous at each point in $\Lambda_*$.
Note that
\begin{equation}\label{ABconvert}
	\A_\nu(t)= \nu\B_{\xi(\nu)}(\nu t)
\qquad\hbox{and}\qquad
	\AA_\nu= \nu\BB_{\xi(\nu)}. 
\end{equation}

Since the map 
\begin{equation}\label{nuK}
 (\nu,K)\in (0,\infty)\times CB(H)\mapsto\nu K \text{ is continuous,}
\end{equation}
the continuity of $Q_4$ in \eqref{Bcont} and the second identity in \eqref{ABconvert} imply that $\nu\mapsto \AA_\nu$ is continuous at each point in $\Lambda_*$. 

\textit{Claim.}  Given $\tau_0\in \R$. If  $g\in L^\infty(\R,H)\mapsto \B_{g}(\tau_0)$ is continuous at $g_0$, then the map $(g,\tau)\mapsto \B_{g}(\tau)$ is continuous at $(g_0,\tau_0)$.

This \textit{Claim} and the continuity of $Q_3$ in \eqref{Bcont} imply that the map $\nu\mapsto \B_{\xi(\nu)}(\nu t)$ is continuous at each point in $\Lambda_*$, for any $t\in\R$. Combining this fact with the first identity in \eqref{ABconvert} and property \eqref{nuK} proves that the map  $\nu\mapsto \A_{\nu}(t)$ is continuous at each point in $\Lambda_*$, for any $t\in\R$.

It remains to prove \textit{Claim}. 
By the triangle inequality, 
\begin{equation}\label{Bht}
 \Delta_H (\B_{g}(\tau),\B_{g_0}(\tau_0))\le  \Delta_H (\B_{g}(\tau),\B_{g}(\tau_0))+ \Delta_H (\B_{g}(\tau_0),\B_{g_0}(\tau_0)). 
\end{equation}
Since $g\to g_0$, we have $g$ belongs to a bounded subset of $L^\infty(\R,H)$.
Then the first term on the right-hand side of \eqref{Bht} goes to zero as $(g,\tau)\to(g_0,\tau_0)$ by the virtue of Proposition \ref{tcont} below.
The second term on right-hand side of \eqref{Bht} goes to zero by the assumption $g\mapsto\B_{g}(\tau_0)$ is continuous at $g_0$.
Thus, the map $(g,\tau)\mapsto \B_{g}(\tau)$ is continuous at $(g_0,\tau_0)$. This finishes the proof of \textit{Claim} and also the proof of this corollary.
\end{proof}

The following result is about the continuity in time of pullback attractors for the Navier--Stokes equations \eqref{NSE2D}. 
It has its own merit, and is stronger than what is needed for the proof in Corollary \ref{nsvisc}.

\begin{proposition}\label{tcont}
 Let $\A_{\nu,f}(t)$ be pullback attractors of the Navier--Stokes equations \eqref{NSE2D} with $\nu\in(0,\infty)$ and $f\in L^\infty(\R,H)$.
 Then the map $t\mapsto \A_{\nu,f}(t)$ is locally H\"older continuous on $\R$, uniformly in $(\nu,f)$ for $\nu$ belonging to any compact subsets of $(0,\infty)$ and $f$ belonging to any bounded subsets of $L^\infty(\R,H)$.
\end{proposition}
\begin{proof}
Denote by $S_{\nu,f}(t,s)$ the process generated by solutions of the Navier--Stokes equations  \eqref{NSE2D}.
Let $R_0>0$ and $\varepsilon_0\in (0,1)$. Define
$\bar G =R_0/(\lambda_1\varepsilon_0^2)$ and 
$$\bar \rho =\max\{\rho(\bar G):\nu\in[\varepsilon_0,\varepsilon_0^{-1}]\}.$$

Let $D=\bar B_V(\bar \rho)$. Consider $\|f\|_{\infty}\le R_0$ and $\nu\in[\varepsilon_0,\varepsilon_0^{-1}]$.
Same arguments as in Theorem \ref{thm42} for the set $K$ replaced by $D$,  and same as \eqref{uc1}, we have 
\begin{equation}\label{Anflim}
 \A_{\nu,f}(t)=\lim_{s\to -\infty} \overline{S_{\nu,f}(t,s)D}.
\end{equation}

Given $u_0\in D$, let $u(t)=S_{\nu,f}(t,s)u_0$. By Theorem \ref{apriori}, there is $T_0>0$ depending on $R_0$ and $\varepsilon_0$ such that if $t-s\ge T_0$ then
\begin{equation}\label{R1bar}
 \|u(t)\|\le R_1= 2\varepsilon_0^{-2}\bar G^2,\qquad \|\nabla u(t)\|\le \bar \rho.
\end{equation}

We recall inequality (2.32) in \cite{Temam} p.~111:
\begin{equation}\label{dd}
\frac{\d}{\d t}\|\nabla u\|^2+\nu |Au|^2\le \frac2\nu \|f\|^2+\frac{2c_0}{\nu^3}\|u\|^2\|\nabla u\|^4, 
\end{equation}
where $c_0>0$ is an appropriate constant.

Consider $t_2>t_1>s+T_0$ with $t_2-t_1<1$. Integrating \eqref{dd} in time from $t_1$ to $t_2$ gives
\begin{align*}
  \nu \int_{t_1}^{t_2} |Au|^2\d\tau \le \|\nabla u(t_1)\|^2 + \frac{2c_0}{\nu^3} \int_{t_1}^{t_2} \|u\|^{2} \|\nabla u\|^4 \d\tau+\frac2\nu \int_{t_1}^{t_2} \|f\|^2\d\tau,
 \end{align*}
Using estimates in \eqref{R1bar} yields
 $$\nu \int_{t_1}^{t_2} |Au|^2\d\tau \le R_2=\bar\rho^2+2c_0 \varepsilon_0^{-3} R_1^2\bar \rho^4+2\varepsilon_0^{-1}R_0^2.$$

Now, 
 \begin{equation}
 u(t_2)-u(t_1)=-\nu \int_{t_1}^{t_2}Au(\tau)\d\tau -\int_{t_1}^{t_2} B(u(\tau),u(\tau))\d\tau+\int_{t_1}^{t_2}f(\tau)\d\tau. 
 \end{equation}
Then,
 \begin{align*}
 \|u(t_2)-u(t_1)\|
 &\le \nu\int_{t_1}^{t_2}\|Au(\tau)\|\d\tau + c_1 \int_{t_1}^{t_2} \|u(\tau)\|^{1/2} \|Au(\tau)\|^{1/2} \|\nabla u(\tau)\|\d\tau + R_0(t_2-t_1)\\
 &\le \nu^{1/2} (t_2-t_1)^{1/2}\Big(\nu \int_{t_1}^{t_2}\|Au(\tau)\|^2\d\tau\Big)^{1//2} \\
&\quad  + c_1 R_1^{1/2}\bar \rho (t_2-t_1)^{3/4}\nu^{-1/4}\Big(\nu\int_{t_1}^{t_2} \|Au(\tau)\|^{2}\d\tau\Big)^{1/4} + R_0(t_2-t_1).
 \end{align*}
 Above, $c_1>0$ is a constant independent of $R_0$, $\varepsilon_0$. Utilizing estimate \eqref{Anflim} gives
  \begin{align*}
 \|u(t_2)-u(t_1)\|
  &\le \varepsilon_0^{-1/2}R_2^{1/2}(t_2-t_1)^{1/2}+ c_1 \varepsilon_0^{-1/4}R_1^{1/2}\bar \rho R_2^{1/4}(t_2-t_1)^{3/4} + R_0(t_2-t_1).
 \end{align*}
 Therefore, there is $M>0$ depending on $R_0$ and $\varepsilon_0$ such that
 $$ \|u(t_2)-u(t_1)\|\le M(t_2-t_1)^{1/2}.$$
 Consequently,
$$  \Delta_H (S_{\nu,f}(t_2,s)D,S_{\nu,f}(t_1,s)D)\le M(t_2-t_1)^{1/2}.$$
By this and \eqref{Anflim},  we have 
 \begin{equation}
  \Delta_H (\A_{\nu,f}(t_2),\A_{\nu,f}(t_1))\le M(t_2-t_1)^{1/2}.
 \end{equation}
This finishes the proof of our proposition.
 \end{proof}

\section*{Appendix A}
This appendix reproduces the formal {\it a priori}\/ estimates
for the two-dimensional incompressible Navier--Stokes equations
stated as Theorem~\ref{apriori} in the main body of the paper.
These are essentially the estimates that appear
in \cite{Temam} pages 109--111 for the time-independent case
that have been adapted for time-dependent body forces
$f\in L^\infty(\R,H)$.

Let $u_0\in H$, $t\ge s$, and set $u(t)=S_f(t,s)u_0$.
Taking the inner product of (\ref{NSEfn}) with $u$ and
using the orthogonality property
\begin{equation}\label{ortho}
\<B(u,v),v\>=0
\end{equation}
we have
\begin{equation}\label{diq1}
\frac{\d}{\d t}\|u\|^2+\nu \|\nabla u\|^2 \le
\frac{\|f\|^2}{\nu\lambda_1}.
\end{equation}
Integrating \eqref{diq1} in time from $s$ to $t$ yields
\begin{equation}
\nu\int_s^t\|\nabla u(\tau)\|^2\,\d\tau \le \|u_0\|^2 +
(t-s)\frac{\|f\|^2_{L^\infty(\R,H)}}{\nu\lambda_1}
\end{equation}
Noting that $\|f\|_{L^\infty(\R,H)}=\nu^2\lambda_1 G$ obtains
\eqref{intV} in Theorem~\ref{apriori}.

By Poincar\'e's inequality we obtain
\begin{equation*}
\frac{\d}{\d t}\|u\|^2+\nu\lambda_1 \|u\|^2 \le
\frac{\|f\|^2}{\nu\lambda_1}.
\end{equation*}
Using Gronwall's inequality yields
\begin{equation}\label{fromG}
\|u(t)\|^2\le \|u_0\|^2 e^{-\nu\lambda_1 (t-s)}
	+\rho_0^2 (1-e^{-\nu\lambda_1 (t-s)})
	\quad\mbox{where}\quad
	\rho_0=\nu G.
\end{equation}
Therefore,
for each bounded subset $B$ of $H$, there exists a time $t_0(B)$
such that
\begin{equation}\label{absH}
\|u(t)\|^2\le 2\rho_0^2\qquad\mbox{for all}\qquad t\ge t_0(B).
\end{equation}
We have, consequently, obtained the first part of \eqref{Vbound}.

Finally, we recall here the needed estimates for $\|\nabla u(t)\|$. (Again, calculations can be found, for example, in \cite{Temam} pages 109--111.)
Let $\nu>0$. Recall from \eqref{fromG} that $\rho_0=\nu G$,
and define 
$$\rho'_{0}=\rho_{0}+1,\quad m_{1}=\lambda_1 \rho_{0}^2+\frac{{\rho'}^2_{0}}{\nu},\quad
m_{2}=2\nu \lambda_1^2 \rho_0^2,\quad
m_{3}=\frac{2c_0}{\nu^3}{\rho'}^2_{0} m_{1},$$
where $c_0$ is the same as in \eqref{dd}.
For  $R>\rho'_{0}$, denote
$$t_{1}(R)=1+\frac{1}{\nu\lambda_1}\log\frac{R^2}{2\rho_{0}+1}.$$
Then, for $\|u_0\|\le R$ and all $t\ge t_{1}(R)$,
\begin{equation}\label{uGt}
\|\nabla u(t+s)\|^2\le \rho(G)\qquad\hbox{where}\qquad
	\rho(G)=(m_{1} +m_{2})e^{m_{3}}.
\end{equation}
Noting that $\rho(G)$ is an increasing function of $G$ such that
also depends on $\nu$ and $\lambda_1$ yields the final part
of \eqref{Vbound} in Theorem~\ref{apriori}.

\section*{Acknowledgments}
LTH acknowledges the support by NSF grant DMS-1412796.
EJO was supported in part by NSF grant DMS grant DMS-1418928.
JCR was supported by an EPSRC Leadership Fellowship EP/G007470/1, as was the visit of EJO to Warwick while on sabbatical leave from UNR.

\bigskip

\noindent\textit{\today}.

\end{document}